\documentclass[a4paper,12pt,reqno]{amsart}

\usepackage[english]{babel}
\usepackage{amsmath}
\usepackage{amsfonts,amssymb,mathrsfs,amscd,amsthm}

\numberwithin{equation}{section}

\theoremstyle{plain}
\newtheorem{theorem}{Theorem}[section]
\newtheorem{propos}[theorem]{Proposition}
\newtheorem{cor}[theorem]{Corollary}
\newtheorem{lem}[theorem]{Lemma}
\newtheorem{quest}[theorem]{Question}
\newtheorem{problem}[theorem]{Problem}
\theoremstyle{definition}

\newtheorem{defin}[theorem]{Definition}
\newtheorem{remark}[theorem]{Remark}

\newcommand{\hM}{\widehat{M}}
\newcommand{\As}{\mathrm{As}}
\newcommand{\Cy}{\mathrm{Cy}}
\newcommand{\Pe}{\mathrm{Pe}}

\newcommand{\Int}{\mathop{\mathrm{Int}}\nolimits}
\newcommand{\Sk}{\mathop{\mathrm{Sk}}\nolimits}
\newcommand{\pr}{\mathop{\mathrm{pr}}\nolimits}
\newcommand{\cone}{\mathop{\mathrm{cone}}\nolimits}
\newcommand{\mvol}{\mathop{\mathrm{vol}}\nolimits}

\newcommand{\K}{\mathbb{K}}
\newcommand{\R}{\mathbb{R}}

\newcommand{\C}{\mathbb{C}}
\newcommand{\Z}{\mathbb{Z}}
\newcommand{\Q}{\mathbb{Q}}

\newcommand{\CU}{\mathcal{U}}

\newcommand{\CC}{\mathcal{C}}
\newcommand{\CB}{\mathcal{B}}
\newcommand{\CI}{\mathcal{I}}
\newcommand{\CR}{\mathcal{R}}

\newcommand{\bK}{\overline{K}}
\newcommand{\clambda}{\lambda_{\text{\textnormal{can}}}}

\begin{document}

\title[Small covers and realization of cycles]{Small covers of graph-associahedra and realization of cycles}

\author{Alexander A.~Gaifullin}

\thanks{The work is supported by the Russian Science Foundation under grant 14-11-00414.}

\address{Steklov Mathematical Institute of the Russian Academy of Sciences}


\email{agaif@mi.ras.ru}

\date{}


\sloppy

\begin{abstract}
An oriented connected closed manifold~$M^n$ is called a \textit{URC-manifold} if for any oriented connected closed manifold~$N^n$ of the same dimension there exists a non-zero degree mapping of a finite-fold covering~$\widehat{M}^n$ of~$M^n$ onto~$N^n$. This condition is equivalent to the following: For any $n$-dimensional integral homology class of any topological space~$X$, a multiple of it can be realized as the image of the fundamental class of a finite-fold covering~$\widehat{M}^n$ of~$M^n$ under a continuous mapping $f\colon \widehat{M}^n\to X$. In 2007 the author gave a constructive proof of the classical result by Thom that a multiple of any integral homology class can be realized as an image of the fundamental class of an oriented smooth manifold. This construction yields the existence of URC-manifolds of all dimensions. For an important class of manifolds, the so-called small covers of graph-associahedra corresponding to connected graphs, we prove that either they or their two-fold orientation coverings are URC-manifolds. In particular, we obtain that the two-fold covering of the small cover of the usual Stasheff associahedron is a URC-manifold. In dimensions~4 and higher, this manifold is simpler than all previously known URC-manifolds.
\end{abstract}

\maketitle

\section{Introduction}
The following question is called Steenrod's problem on realization of cycles: Given an integral homology class~$z\in H_n(X;\Z)$ of a topological space~$X$, does there exist an oriented closed smooth manifold~$N^n$ and a  continuous mapping~$f$ of~$N^n$ to~$X$ such that $f_*[N^n]=z$? If such pair $(N^n,f)$ exists, the class~$z$ is said to be  \textit{realizable by Steenrod}. A classical theorem of Thom~\cite{Tho58} asserts that: 
\begin{itemize}
\item All classes of dimensions less than or equal to~$6$ are realizable by Steenrod.
\item In every dimension greater than or equal to~$7$, there exist non-realizable classes.
\item Any integral homology class~$z$ of any topological space is realizable by Steenrod with some multiplicity, that is, there exists an integer $k>0$ such that the class~$kz$ is realizable by Steenrod. (In fact, the number~$k$ depends only on the dimension~$n$.)
\end{itemize} 

Novikov~\cite{Nov62} showed that a class~$z\in H_n(X;\Z)$ is realizable by Steenrod if for all $j\ge 1$ and all primes~$p>2$ the  $(n-2j(p-1)-1)$-dimensional homology of~$X$ does not contain $p$-torsion. The best known bound for the multiplicity~$k$ was obtained by Buchstaber~\cite{Buc69},~\cite{Buc70}:
$$
k\le\prod_{p>2}p^{\left[\frac{n-2}{2(p-1)}\right]}\,,
$$
where the product is taken over all odd primes~$p$.

Notice that the question about manifolds that realize homology classes was not discussed in~\cite{Tho58}, \cite{Nov62}, \cite{Buc69}, \cite{Buc70}, since the methods of these papers did not give any approach to it. 

The author~\cite{Gai07} gave a constructive solution of Steenrod's problem, i.\,e., presented an explicit combinatorial construction of a pair~$(N^n,f)$ that realizes a multiple of the given homology class. This construction does not allow to obtain effective estimates for the multiplicity~$k$. But, unlike the classical algebraic topological approach, it allows us to control the topology of the obtained manifold~$N^n$. Namely, it was shown in~\cite{Gai08a},~\cite{Gai08b} that in every dimension~$n$ there exists a single manifold~$M_0^n$ \ --- \ the so-called Tomei manifold --- such that for any homology class~$z\in H_n(X;\Z)$ of any topological space~$X$, the manifold~$N^n$ that realizes a multiple~$kz$ of~$z$ can be chosen to be a (non-ramified) finite-fold covering of~$M_0^n$.  
The Tomei manifold~$M^n_0$ first appeared in the paper~\cite{Tom84} by Tomei as the isospectral manifold of symmetric tridiagonal real matrices. The problem arises: Which other manifolds~$M^n$, except for the Tomei manifold~$M^n_0$, satisfy the same universal property with respect to the problem on realization of cycles? The author~\cite{Gai13},~\cite{Gai13-b} has introduced the following class of URC-manifolds (form the words ``Universal Realizator of Cycles'').

\begin{defin}
An oriented connected closed manifold~$M^n$ is called a \textit{URC-manifold} if, for any topological space~$X$ and any homology class $z\in H_n(X;\Z)$, there exists a finite-fold covering~$\hM^n$ of~$M^n$ and a continuous mapping $f\colon\hM^n\to X$ such that $f_*[\hM^n]=kz$ for a non-zero integer~$k$. 
\end{defin}

\begin{remark}
We always endow the covering~$\hM^n$ of~$M^n$ with the orientation induced by the orientation of~$M^n$, i.\,e., such that the Jacobian of the projection  $\hM^n\to M^n$ is positive at all points of~$\hM^n$. We do not require~$\hM^n$ to be connected. If we restrict ourselves to considering pathwise connected spaces~$X$ only and require~$\hM^n$ to be connected, then we obtain an equivalent definition of a URC-manifold
\end{remark}

Another motivation for the class of URC-manifolds is the problem about the domination relation on the set of homotopy types of oriented closed manifolds. Let~$M^n$ and~$N^n$ be connected oriented closed manifolds of the same dimension. We say that $M^n$ \textit{dominates\/}~$N^n$ and write $M^n\ge N^n$ if there exists a non-zero degree mapping of~$M^n$ onto~$N^n$. We say that~$M^n$ \textit{virtually dominates\/}~$N^n$ if a finite-fold covering of~$M^n$ dominates~$N^n$. The study of the domination relation goes back to the works of Milnor and Thurston~\cite{MiTh77} and Gromov~\cite{Gro82}. Carlson and Toledo~\cite{CaTo89} posed a problem of finding of a reasonable \textit{maximal class\/} of $n$-dimensional manifolds with respect to domination, that is, a class of  $n$-dimensional manifolds such that any $n$-dimensional manifolds can be dominated by a manifold in this class.  Each URC-manifold~$M^n$ provides a solution to this problem: For a required class one can take the class of all finite-fold coverings of~$M^n$. Equivalently, URC-manifolds are exactly manifolds that are greatest with respect to virtual domination. In other words, a manifold~$M^n$ is a URC-manifold if and only if it virtually dominates every oriented connected closed manifold of the same dimension. A good survey of works on the domination relation can be found in~\cite{KoLo09}.

The first example of a URC-manifold, the Tomei manifold~$M_0^n$ mentioned above, is a \textit{small cover\/} of certain simple polytope, namely, of the \textit{permutahedron\/}~$\Pe^n$. This means that~$M_0^n$ is glued in a special way out of $2^n$ copies of~$\Pe^n$. (Strict definitions of a small cover and of the permutahedron will be given in Section~\ref{section_defin}.) In~\cite{Gai13} the author has found many other examples of URC-manifolds among small covers of other simple polytopes. However, all these simple polytopes, except for the dodecahedron in dimension~$3$, are more complicated than the permutahedron. Here `more complicated' can be understood, for instance, in the sense that they have greater numbers of facets. Hence, in dimensions~$4$ and higher, the Tomei manifolds~$M_0^n$ have remained the simplest known URC-manifolds.

The main result of the present paper is the construction of new examples of URC-manifolds that are simpler than the Tomei manifolds. Carr and Devados~\cite{CaDe05} suggested a construction that assigns a simple polytope~$P_{\Gamma}$, which they called a \textit{graph-associahedron\/},  to every finite graph~$\Gamma$, see Section~\ref{subsection_as} for details. Our main result is as follows: For any connected graph~$\Gamma$, any orientable small cover~$M_{P_{\Gamma},\lambda}$ of~$P_{\Gamma}$ is a URC-manifold, and for any non-orientable small cover~$M_{P_{\Gamma},\lambda}$ of~$P_{\Gamma}$, the two-fold orientation covering of~$M_{P_{\Gamma},\lambda}$ is a  URC-manifold. This result will be formulated in details in Section~\ref{subsection_result} after we give all necessary definitions. We shall present an explicit construction that, for a singular cycle representing the given homology class, builds a manifold being a finite-fold covering of~$M_{P_{\Gamma},\lambda}$ that realizes a multiple of the given homology class. This construction is a further modification of the construction suggested by the author in~\cite{Gai07} and developed in~\cite{Gai08a}, \cite{Gai08b}, and~\cite{Gai13}. Notice that another modification of this construction was applied to solve the problem of constructing of an oriented simplicial manifold with the prescribed set of links of vertices, see~\cite{Gai08c}.

\begin{remark}
Unfortunately, we shall always face the following inconvenience concerning the terminology. We use to work with coverings of manifolds, and the word `covering' is always understood in the sense of classical topology, i.\,e., as a non-ramified covering. On the other hand, the main source of example of manifolds for us is the construction  of small covers of simple polytopes. Here the expression `small cover' should be understood as a single whole, and a small cover of a polytope is by no means a covering of this polytope in the sense of classical topology. Nevertheless, the term `small cover'  is established and it is impossible to change it. Hence, we shall often deal with objects like a `two-fold covering of a small cover of a simple polytope'. To avoid confusion, we settle for all that the word `covering' is always understood in the sense of classical topology unless it enters the phrase `small cover'.
\end{remark}

The `smallest' of graph-associahedra corresponding to connected graphs is the well-known \textit{Stasheff associahedron}~$\As^n$. Here the word `smallest' can be understood in several senses. For instance,  among all graph-associahedra corresponding to connected graphs, the Stasheff associahedron has the smallest numbers of faces, and also $h$-numbers and $\gamma$-numbers, see~\cite{BuVo11}. For example, the number of facets of~$\Pe^n$ is equal to $2^{n+1}-2$, while the number of facets of~$\As^n$ is equal to $n(n+3)/2$. One of the URC-manifolds that will be constructed in this paper is the two-fold orientation covering~$\overline{M}_{\As^n,\clambda}$ of the small cover~$M_{\As^n,\clambda}$ of the Stasheff associahedron~$\As^n$ that corresponds to the canonical Delzant characteristic function~$\clambda$, see Section~\ref{section_defin} for details. This manifold is much `smaller' than the Tomei manifold, for instance, in the sense that it has smaller Betti numbers. Notice that the Betti numbers with coefficients in~$\Z_2$ of a small cover~$M_{P,\lambda}$ are independent of the choice of the  characteristic function~$\lambda$ and are equal to the $h$-numbers of~$P$, see~\cite{DaJa91}. The Betti numbers with coefficients in~$\Q$ of small covers are harder to compute, see Section~\ref{subsection_rational}. Up to now, in every dimension~$n\ge 4$, the manifold~$\overline{M}_{\As^n,\clambda}$ is the simplest known URC-manifold.

This paper is organized in the following way. In Section~\ref{section_defin} we give all necessary definitions and then formulate our main results, Theorem~\ref{theorem_main} and Corollary~\ref{cor_main}, which claim that for any graph-associahedron~$P_{\Gamma}$ corresponding to a connected graph~$\Gamma$ the following manifolds are URC-manifolds:
\begin{itemize}
\item the real moment-angle manifold over~$P_{\Gamma}$, 
\item all orientable small covers of~$P_{\Gamma}$,
\item the two-fold orientation coverings of all non-orientable small covers of~$P_{\Gamma}$. 
\end{itemize}
An explicit construction for realization of cycles providing the proofs of these results is contained in Sections~\ref{section_simplex}--\ref{section_proof}. Finally, in Section~\ref{section_small} we discuss how one can compare different URC-manifolds, and in what senses the newly constructed URC-manifolds are smaller than URC-manifolds known before.

\section{Definitions and main result}\label{section_defin}

\subsection{Small covers and real moment-angle manifolds} Recall that an  $n$-dimensional convex polytope $P\subset\R^n$ is said to be \textit{simple} if every its vertex is contained in exactly $n$ facets.

An important branch of modern algebraic geometry is theory of \textit{toric varieties}. Recall that a toric variety is a normal algebraic variety that contains a Zariski open subset isomorphic to the algebraic torus~$(\C^{\times})^n$ such that the action of this torus on itself extends to its action on the whole variety. It is well known that, for a projective toric variety~$X$, the quotient of~$X$ by the action of the compact torus~$T^n\subset(\C^{\times})^n$ is a simple polytope~$P^n$. Davis and Januszkiewicz~\cite{DaJa91} suggested a construction of topological analogs of toric varieties, i.\,e., of smooth topological manifolds~$M^{2n}$ with locally standard actions of the half-dimensional compact torus~$T^n$
such that $M^{2n}/T^n=P^n$ is a simple polytope. Today such manifolds are called \textit{quasi-toric manifolds}. Alongside with this construction, in~\cite{DaJa91} they also introduced a real version of it with the torus~$T^n$ replaced by its real analog, that is, by the group~$\Z_2^n$. The obtained manifolds are called \textit{small covers} of simple polytopes. (Here and further, we denote by~$\Z_2$ the cyclic group of order~$2$ and, unlike in~\cite{DaJa91}, use the additive notation for it.) In the construction of quasi-toric manifolds and small covers, an auxiliary role was played by certain special manifolds~$\mathcal{Z}_P$ and~$\CR_P$ with the actions of the groups~$T^m$ and~$\Z_2^m$, respectively, such that $\mathcal{Z}_P/T^m=P^n$ and $\CR_P/\Z_2^m=P^n$, where $m$ is the number of facets of~$P^n$.  Theory of these manifolds was developed in the works by Buchstaber and Panov, see~\cite{BuPa15}. It turned out that they have significant independent importance. Today they are called \textit{moment-angle manifolds} and \textit{real moment-angle manifolds}, respectively.
Below, among all constructions mentioned above, we shall need only the constructions of real moment-angle manifolds and small covers. Hence, we shall adduce only these two constructions. 

Let $P$ be an $n$-dimensional simple convex polytope with $m$ facets $F_1,\ldots,F_m$.  Consider the group~$\Z_2^m$ with the standard basis $a_1,\ldots,a_m$. We put,
$$
\CR_{P}=(P\times\Z_2^m)/\sim\,,
$$
where $\sim$ is the equivalence relation on~$P\times\Z_2^m$ such that $(x,g)\sim (x',g')$ if and only if $x=x'$ and the element $g-g'$ belongs to the subgroup of~$\Z_2^m$ generated by all~$a_i$ such that $x\in F_i$. We denote by~$[x,g]$ the point of~$\CR_P$ corresponding to the equivalence class of~$(x,g)$.

Now, consider the group~$\Z_2^n$. A homomorphism $\lambda\colon\Z_2^m\to\Z_2^n$ is called a \textit{characteristic function} if, for each vertex~$v$ of~$P$, the elements $\lambda(a_{i_1}),\ldots,\lambda(a_{i_n})$ corresponding to the facets $F_{i_1},\ldots,F_{i_n}$ containing~$v$ form a basis of~$\Z_2^n$. Recall that not every simple polytope admits a  characteristic function. To a pair $(P,\lambda)$ such that $P$ is a simple polytope and $\lambda$ is a characteristic function, one can assign the manifold
$$
M_{P,\lambda}=(P\times\Z_2^n)/\sim_{\lambda},
$$
where $\sim_{\lambda}$ is the equivalence relation on~$P\times\Z_2^n$ such that $(x,g)\sim_{\lambda} (x',g')$ if and only if  $x=x'$ and the element $g-g'$ belongs to the subgroup of~$\Z_2^n$ generated by all~$\lambda(a_i)$ such that $x\in F_i$. We denote by $[x,g]_{\lambda}$ the point of~$M_{P,\lambda}$ corresponding to the equivalence class of~$(x,g)$.

The manifold $\CR_{P}$ is called the  \textit{real moment-angle manifold\/} over the simple polytope~$P$, and the manifolds~$M_{P,\lambda}$ are called \textit{small covers\/} of~$P$. The manifold~$\CR_P$ is glued out of $2^m$ copies of~$P$ indexed by the elements of~$\Z_2^m$. Each manifold~$M_{P,\lambda}$ is glued out of $2^n$ copies of~$P$ indexed by the elements of~$\Z_2^n$. The copies of~$P$ corresponding to the elements~$g$ and~$g'$ are glued to each other along a facet~$F_i$ if and only if $g-g'=a_i$ for~$\CR_P$ and $g-g'=\lambda(a_i)$ for~$M_{P,\lambda}$. Herewith, the joining of the simple polytopes at every point in either~$\CR_P$ or~$M_{P,\lambda}$ is locally modeled by the joining of the coordinate orthants at a point in~$\R^n$. Therefore, $\CR_P$ and $M_{P,\lambda}$ are indeed topological manifolds. Any simple polytope~$P$ has the standard structure of a smooth manifold with corners, which induces smooth structures on the manifolds~$\CR_P$ and~$M_{P,\lambda}$. It is easy to see that, for any characteristic function~$\lambda$, the mapping $p_{\lambda}\colon \CR_P\to M_{P,\lambda}$ given by $[x,g]\mapsto [x,\lambda(g)]_{\lambda}$ is an $2^{m-n}$--fold regular covering. In other words, $M_{P,\lambda}=\CR_P/\ker\lambda$, where the group~$\Z_2^m$ and, hence, the subgroup~$\ker\lambda$ of it acts on~$\CR_P$ by $h\cdot [x,g]=[x,g+h]$.

A simple polytope~$P$ is called a \textit{flag polytope} if every set $F_1,\ldots,F_k$ of its pairwise intersecting facets has a non-empty intersection~$F_1\cap\cdots\cap F_k$. It follows from the results of Davis~\cite{Dav83} that small covers~$M_{P,\lambda}$ and the real moment-angle manifold~$\CR_P$ over a flag polytope~$P$ are aspherical manifolds, that is,  $\pi_i(M_{P,\lambda})=0$ and $\pi_i(\CR_P)=0$ whenever $i>1$. (The manifolds~$M_{P,\lambda}$ and~$\CR_P$ had not  been introduced in~\cite{Dav83},  but the manifold~$\CU_P$, which is the universal covering of them, was studied and was proved to be contractible.)

\begin{remark}
If one endows a simple polytope~$P$ with a natural structure of an orbifold, and defines properly the concept of the fundamental group of an orbifold (cf.~\cite[Section~13.2]{Thu02}), then the manifold~$\CR_P$ will turn out a \textit{universal Abelian covering} of the orbifold~$P$, i.\,e., the covering corresponding to the commutant of its fundamental group. Hence the manifold~$\CR_P$ is sometimes called the universal Abelian covering of~$P$. We prefer not to use this terminology, since in the present paper the word `covering' is already overloaded. The fundamental group of~$P$ treated as an orbifold is a right-angular Coxeter group. The commutants of such groups have been studied by Panov and Veryovkin~\cite{PaVe16}.  
\end{remark}

\subsection{Nestohedra and graph-associahedra}\label{subsection_as} 
We shall need to consider an important family of simple polytopes, which are called \textit{graph-associahedra}. These polytopes were introduced by Carr and Devados~\cite{CaDe05}. Their definition is based on the fundamental concept of a  \textit{building set}, which goes back to the paper~\cite{DCPr95} by De Concini and Procesi on the models of the complements of subspace arrangements in a vector space. The same polytopes were studied by Toledano Laredo~\cite{TL08} under the name \textit{De Concini--Procesi associahedra}. Graph-associahedra are representatives of an important wider class of simple polytopes called  \textit{nestohedra} and introduced in~\cite{FeSt05},~\cite{Pos09}.
(The term `nestohedron' was first used in~\cite{PRW08}.)

Let $V$ be a finite set. A \textit{building set} on~$V$ is a subset~$\CB$  of non-empty subsets $S\subseteq V$ satisfying the following conditions:
\begin{enumerate}
\item If $S_1, S_2\in \CB$ and $S_1\cap S_2\ne\varnothing$, then $S_1\cup S_2\in\CB$.
\item All one-elements subsets $\{i\}$, $i\in V$, belong to~$\CB$.
\end{enumerate}
A building set~$\CB$ is called \textit{connected\/} if $V\in\CB$.

In the present paper under a  \textit{graph\/} we always mean a finite simple graph, i.\,e., a finite graph without loops and multiple edges. Let  $\Gamma$ be a graph on the vertex set~$V$. For each subset $S\subseteq V$, we denote by~$\Gamma|_S$ the \textit{restriction\/} of~$\Gamma$ to~$S$, that is, the graph on the vertex set~$S$, such that every pair of vertices  $s_1,s_2\in S$ is connected by an edge in~$\Gamma|_S$ if and only if it is connected by an edge in~$\Gamma$. The \textit{graph building set\/} corresponding to~$\Gamma$ is the set~$\CB(\Gamma)$ consisting of all non-empty subsets $S\subseteq V$ such that the graph~$\Gamma|_{S}$ is connected. It is easy to see that conditions~1 and~2 are satisfied, that is, $\CB(\Gamma)$ is indeed a building set. Besides, the building set~$\CB(\Gamma)$ is connected if and only if the graph~$\Gamma$ is connected.

Recall that the  \textit{Minkowski sum\/} of subsets $P$ and~$Q$ of~$\R^k$ is the subset $P+Q$ of~$\R^k$ consisting of all vectors $p+q$ such that $p\in P$ and $q\in Q$. 

Consider the space~$\R^{n+1}$ with the basis $e_0,\ldots,e_n$ and with the coordinates $x_0,\ldots,x_n$ in this basis. The \textit{standard $n$-dimensional simplex\/} is the simplex $\Delta^n\subset\R^{n+1}$ with vertices $e_0,\ldots,e_n$. Equivalently, $\Delta^n$ is the set of all points $(x_0,\ldots,x_n)$ satisfying $\sum_{i=0}^nx_i=1$ and $x_i\ge0$ for all~$i$. For each non-empty subset $S\subseteq\{0,\ldots,n\}$, we denote by~$\Delta_S$ the face of~$\Delta^n$ spanned by all vertices~$e_i$ such that $i\in S$.

Let $\CB$ be a building set on the set $V=\{0,\ldots,n\}$. The \textit{nestohedron} corresponding to~$\CB$ is the polytope
$$
P_{\CB}=\sum_{S\in \CB}\Delta_S,
$$
where the sum is the Minkowski sum.  If $\CB=\CB(\Gamma)$ is the graph building set corresponding to a graph~$\Gamma$, then the nestohedron $P_{\Gamma}=P_{\CB(\Gamma)}$ is called a \textit{graph-associahedron}.  

In the following proposition we collect basic results of papers~\cite{FeSt05}, \cite{Pos09}, \cite{PRW08} on nestohedra in the interesting to us partial case of graph-associahedra corresponding to connected graphs.

\begin{propos} Let $\Gamma$ be a connected graph on the vertex set $V=\{0,\ldots,n\}$. Then the graph-associahedron~$P_{\Gamma}$ satisfies the following:
\begin{enumerate}
\item $P_{\Gamma}$ is an $n$-dimensional flag simple polytope lying in the hyperplane  $H\subset\R^{n+1}$ given by $\sum_{i=0}^nx_i=|\CB(\Gamma)|$.
\item  $P_{\Gamma}$ has $|\CB(\Gamma)|-1$ facets, which can be naturally indexed by the subsets $S\in\CB(\Gamma)\setminus\{V\}$. The facet~$F_S$ corresponding to a subset~$S$ lies in the intersection of~$H$ and the hyperplane given by $\sum_{i\in S}x_i=k_S$, where $k_S$ is the number of subsets $T\in\CB(\Gamma)$ such that $T\subseteq S$. 
\item Facets~$F_S$ and~$F_T$ intersect if and only if  $S\subseteq T$ or $T\subseteq S$ or $S\cap T=\varnothing$ and~$S\cup T\notin \CB(\Gamma)$.
\end{enumerate}
\end{propos} 

We shall conveniently put~$\CB'(\Gamma)=\CB(\Gamma)\setminus\{V\}$. Then the facets of the graph-associahedron~$P_{\Gamma}$ are indexed by the elements of~$\CB'(\Gamma)$.

It is not hard to see that, for a disconnected graph~$\Gamma$ with connected components $\Gamma_1,\ldots,\Gamma_q$, the graph-associahedron~$P_{\Gamma}$ is combinatorially equivalent (and even isometric) to the direct product $P_{\Gamma_1}\times\cdots\times P_{\Gamma_q}$.

The most important are the following three series of graph-associahedra:

1. If $L_{n+1}$ is the path graph on the vertex set $\{0,\ldots,n\}$, i.\,e., the graph with the $n$ edges $\{0,1\}$, $\{1,2\},\ldots,$ $\{n-1,n\}$, then  $P_{L_{n+1}}$ is the usual \textit{associahedron\/}~$\As^n$, which is also called the  \textit{Stasheff polytope}. 

2. If $C_{n+1}$ is the cycle graph on the vertex set  $\{0,\ldots,n\}$,  i.\,e., the graph with the $n+1$ edges $\{0,1\}$, $\{1,2\},\ldots,$ $\{n-1,n\}$, $\{n,0\}$, then $P_{C_{n+1}}$ is  the  \textit{cyclohedron\/}~$\Cy^n$, which is also called the \textit{Bott--Taubes polytope}. 

3. If $K_{n+1}$ is the complete graph on the vertex set $\{0,\ldots,n\}$, then  $P_{K_{n+1}}$ is the \textit{permutohedron\/}~$\Pe^n$. 

\begin{remark}
Usually under a standard permutohedron one means the convex hull of the $(n+1)!$ points obtained by permutations of coordinates of the point $(1,2,3,\ldots,n+1)$. Nevertheless, the above permutohedron~$\Pe^n=P_{K_{n+1}}$ is the convex hull of the  $(n+1)!$ points obtained by permutations of coordinates of the point  $(1,2,4,\ldots,2^n)$. These two polytopes are combinatorially equivalent.
\end{remark}

For a connected graph~$\Gamma$, we construct a mapping $\pi_{\Gamma}\colon P_{\Gamma}\to\Delta^n$ in the following way. Let $K_{\Gamma}$ be the barycentric subdivision of~$P_{\Gamma}$. We map the barycentre of every face $F=F_{S_1}\cap\cdots\cap F_{S_k}$ of~$P_{\Gamma}$ to the barycentre of the face~$\Delta_{V\setminus(S_1\cup\cdots\cup S_k)}$ of~$\Delta^n$. In particular, we map the barycentre of~$P_{\Gamma}$ to the barycentre of~$\Delta^n$. (Notice that if $F_{S_1}\cap\cdots \cap F_{S_k}\ne\varnothing$, then the set $S_1\cup\cdots\cup S_k$ either coincides with one of the subsets~$S_i$ or does not belong to~$\CB(\Gamma)$, hence, never coincides with the whole set~$V$.) Further, we extend the mapping linearly to every simplex of~$K_{\Gamma}$, and denote the obtained mapping by~$\pi_{\Gamma}$.

\begin{propos}\label{propos_pi}
The mapping~$\pi_{\Gamma}$ satisfies the following:
\begin{enumerate}
\item $\pi_{\Gamma}(F_S)\subseteq\Delta_{V\setminus S}$ for all $S\in\CB'(\Gamma)$.

\item $\pi_{\Gamma}(\partial P_{\Gamma})=\partial\Delta^n$.

\item The mapping $\pi_{\Gamma}$ has degree~$1$, that is, takes the fundamental homology class of the pair~$(P_{\Gamma},\partial P_{\Gamma})$ to the fundamental homology class of the pair~$(\Delta^n,\partial\Delta^n)$.
\end{enumerate}
\end{propos}

\begin{proof}
Property~1 follows immediately from the construction of~$\pi_{\Gamma}$. The inclusion~$\pi_{\Gamma}(\partial P_{\Gamma})\subseteq\partial\Delta^n$ follows from property~1. 

Let us prove property~3. It follows immediately from the construction that $\pi_{\Gamma}$  is a simplicial mapping of the barycentric subdivision of~$P_{\Gamma}$ to the barycentric subdivision of~$\Delta^n$, that is, $\pi_{\Gamma}$ maps every simplex of~$K_{\Gamma}$ linearly onto a simplex of the barycentric subdivision of~$\Delta^n$. Obviously, there exist subsets $S_1\subset \cdots\subset S_n\subset V$ such that $|S_i|=i$ and the graphs~$\Gamma|_{S_i}$ are connected, i.\,e., $S_i\in\CB'(\Gamma)$. Denote by~$u_0$ the barycentre of the graph-associahedron~$P_{\Gamma}$, and denote by $u_1,\ldots,u_n$ the barcentres of its faces $F_{S_1}$, $F_{S_1}\cap F_{S_2},\ldots$, $F_{S_1}\cap\cdots\cap F_{S_n}$, respectively. Denote by~$v_0$ the barycentre of the simplex~$\Delta^n$, and denote by $v_1,\ldots,v_n$ the barycentres of its faces $\Delta_{V\setminus S_1},\ldots,\Delta_{V\setminus S_n}$, respectively. Then $\pi_{\Gamma}(u_i)=v_i$, $i=0,\ldots,n$, hence,  $\pi_{\Gamma}$ maps isomorphically the simplex~$\sigma$ with vertices $u_0,\ldots,u_n$ onto the simplex~$\tau$ with vertices $v_0,\ldots,v_n$. Besides, it can be checked immediately that this affine isomorphism preserves the orientation. (The simplex~$\Delta^n$  and the graph-associahedron~$P_{\Gamma}$ lie in parallel hyperplanes, hence, the standard orientation of~$\Delta^n$ induces the orientation of~$P_{\Gamma}$.) Therefore, to prove that~$\pi_{\Gamma}$ has degree~$1$, we suffice to show that no other $n$-dimensional simplex~$\sigma'$ of~$K_{\Gamma}$ is mapped isomorphically onto~$\tau$. Assume the converse, that is, assume that $K_{\Gamma}$ contains a simplex~$\sigma'\ne\sigma$  with vertices $u_0',\ldots,u_n'$ such that $\pi_{\Gamma}(u_i')=v_i$ for all~$i$. Then $u_0'=u_0$ and there exist facets $F_{S_1'},\ldots,F_{S_n'}$ of~$P_{\Gamma}$ such that $u_i'$ is the barycentre of~$F_{S_1'}\cap\cdots\cap F_{S_i'}$ for $i=1,\ldots,n$. Since $\pi_{\Gamma}(u_i')=v_i$, we see that $S_1'\cup\cdots\cup S_i'=S_i$ for all~$i$. In particular, $S_1'=S_1$. Take the smallest~$i$ such that $S_i'\ne S_i$. Then $S_{i-1}'\cup S_i'=S_i\in\CB(\Gamma)$ and neither of the sets~$S_{i-1}'=S_{i-1}$ and~$S_i'$ is contained in the other. Hence, the facets~$F_{S_{i-1}'}$ and~$F_{S_i'}$ do not intersect, which is impossible. Therefore, $S_i'=S_i$ for all~$i$, that is, $\sigma'=\sigma$, which completes the proof of property~3. 

Since the degree of~$\pi_{\Gamma}$ is non-zero, we obtain that the inclusion $\pi_{\Gamma}(\partial P_{\Gamma})\subseteq\partial\Delta^n$ is not strict.
\end{proof}

Recall that an $n$-dimensional simple polytope $P\subset\R^n$ is called a \textit{Delzant polytope} if, for each its vertex~$p$, there exist integral normal vectors to the facets of~$P$ containing~$p$ that form a $\Z$-basis of the standard lattice $\Z^n\subset\R^n$. For each $n$-dimensional Delzant polytope~$P$ with $m$ facets $F_1,\ldots,F_m$, there is a canonical characteristic function $\clambda\colon\Z_2^m\to\Z_2^n$ such that, for each $i$, the value~$\clambda(a_i)$ is the primitive integral normal vector to~$F_i$ reduced modulo~$2$. To each Delzant polytope~$P$ is assigned a smooth projective toric variety, and the small cover~$M_{P,\clambda}$ corresponding to the  canonical characteristic function described above is the set of real points of this projective variety, see~\cite{BuPa15}.

For any connected graph~$\Gamma$ on the vertex set $V=\{0,\ldots,n\}$, the corresponding graph-associahedron~$P_{\Gamma}$ becomes Delzant if one identifies the hyperplane~$H$ containing~$P_{\Gamma}$ with the space~$\R^n$ spanned by the vectors $e_1,\ldots,e_n$ by means of the coordinate projection $\R^{n+1}\to\R^n$. Since facets of~$P_{\Gamma}$ are indexed by the elements~$S\in\CB'(\Gamma)$, we shall conveniently denote the basis element of~$\Z_2^m=\Z_2^{|\CB'(\Gamma)|}$ corresponding to the facet~$F_S$ by~$a_S$. Then the canonical characteristic function $\clambda\colon\Z_2^m\to\Z_2^n$ yielded by the Delzant structure on~$P_{\Gamma}$ is given by
$$
\clambda(a_S)=\left\{
\begin{aligned}
&\sum_{i\in S}b_i&&\text{if $0\notin S$,}\\
&\sum_{i\in V\setminus S}b_i&&\text{if $0\in S$,}
\end{aligned}
\right.
$$
where $b_1,\ldots,b_n$ is the standard basis of~$\Z_2^n$.  

Notice that graph-associahedra may admit other characteristic functions. For instance, the above mentioned Tomei manifold~$M^n_0$ is the small cover of the permutohedron~$\Pe^n=P_{K_{n+1}}$ corresponding to the characteristic function~$\lambda_0$ given by
$\lambda_0(a_S)=b_{|S|}$. 

\subsection{Main result} \label{subsection_result}

\begin{theorem}\label{theorem_main}
For any connected graph~$\Gamma$, the real moment-angle manifold~$\CR_{P_{\Gamma}}$  is a URC-manifold.
\end{theorem}

We shall prove this theorem in Section~\ref{section_proof} after we describe two auxiliary constructions necessary for it in Sections~\ref{section_simplex} and~\ref{section_constr}.

If a graph $\Gamma$ is the disjoint union of two graphs~$\Gamma_1$ and~$\Gamma_2$, then $P_{\Gamma}=P_{\Gamma_1}\times P_{\Gamma_2}$ and $\CR_{P_{\Gamma}}=\CR_{P_{\Gamma_1}}\times \CR_{P_{\Gamma_2}}$. 
In particular, adding an isolated vertex to a graph, we do not change the polytope~$P_{\Gamma}$. Kotschick and L\"oh~\cite{KoLo09} proved that in every dimension  $n\ge 2$ there exist many examples of oriented closed manifolds, e.\,g., all manifolds on strictly negative curvature, that cannot be dominated by a product of two manifolds of positive dimensions. Hence, if $\Gamma$ is a disconnected graph containing at least two connected components that are not isolated vertices, then $\CR_{P_{\Gamma}}$ is not a URC-manifold.

It follows immediately from the definition of a URC-manifold that if  $M_1^n$ and~$M_2^n$ are oriented connected closed manifolds and $M_1^n$ is a finite-fold covering of~$M_2^n$, then $M_1^n$ is a URC-manifold if and only if $M^n_2$ is a URC-manifold.

\begin{cor}\label{cor_main}
Suppose that $\Gamma$ is a connected graph and $\lambda$ is a characteristic function for~$P_{\Gamma}$. Then the small cover~$M_{P_{\Gamma},\lambda}$ is a URC-manifold whenever it is orientable, and if $M_{P_{\Gamma},\lambda}$ is non-orientable, then the two-fold orientation covering~$\overline{M}_{P_{\Gamma},\lambda}$ of it is a URC-manifold.
\end{cor}

It is easy to check that the small cover~$M_{P,\lambda}$ is orientable if and only if  $\lambda(a_{i_1})+\cdots+\lambda(a_{i_{2k+1}})\ne 0$ for every set of an odd number of indices $i_1,\ldots,i_{2k+1}$. This  implies immediately that, for any connected graph~$\Gamma$ with at least three vertices, the manifold~$M_{P_{\Gamma},\clambda}$ is non-orientable.

\begin{cor}
Suppose that $\Gamma$ is a connected graph with at least three vertices. Then the manifold~$\overline{M}_{P_{\Gamma},\clambda}$ is a URC-manifold.
\end{cor}

\section{A subdivision corresponding to a finite graph of a simplicial cell pseudo-manifold}\label{section_simplex}

In this section we give a construction of a special subdivision of a simplicial cell pseudo-manifold. This subdivision will be used in Section~\ref{section_proof} in the proof of Theorem~\ref{theorem_main}.

\begin{defin}
A (\textit{finite}) \textit{simplicial cell complex\/} is a quotient of the disjoint union of a finite set of simplices~$\Delta_1,\ldots,\Delta_N$ (possibly, of different dimensions) by an equivalence relation~$\sim$ such that
\begin{enumerate}
\item The relation~$\sim$ identifies no two distinct points of the same simplex~$\Delta_i$.
\item If $x\in
\Delta_i$, $x'\in \Delta_j$, and $x\sim x'$, then the relation~$\sim$ identifies a face $F\subseteq
\Delta_i$ containing~$x$ with a face $F'\subseteq \Delta_j$ containing~$x'$ by an affine isomorphism. (The simplex is also supposed to be a face of itself.)
\end{enumerate}
The images of faces of the simplices~$\Delta_i$ under the quotient mapping are called  {\it simplices} or {\it faces} of the obtained simplicial cell complex.
\end{defin}

The difference between a simplicial cell complex and a simplicial complex consists in the fact that in a simplicial complex the intersection of any two simplices is either empty or is a face of either simplex, while in a simplicial cell complex two simplices can have several common faces. For instance, a decomposition of a circle into two arcs is a simplicial cell complex but  not a simplicial complex. However, notice that the standard cell decomposition of a circle with a single one-dimensional cell is not a simplicial cell complex, since condition 1 is violated. All necessary facts about simplicial cell complexes can be found in~\cite{BuPa04}.

The set of vertices of a simplicial cell complex~$Z$ will be denoted by~$V(Z)$. The subcomplex of~$Z$ consisting of all its simplices of dimensions less than or equal to~$k$ is called the $k$-\textit{skeleton} of~$Z$ and is denoted by~$\Sk^k(Z)$. A  \textit{regular colouring\/} of vertices of a simplicial cell complex~$Z$ in colours in a set~$A$ is a mapping $C\colon V(Z)\to A$ such that $C(u)\ne C(v)$ for every two vertices~$u$ and~$v$ connected by an edge. For a simplex~$\rho$ of~$Z$ we denote by~$C(\rho)$ the set of colours of vertices of~$\rho$. For~$A$ we shall often take the vertex set~$V(\Gamma)$ of a finite graph~$\Gamma$. Hence colours of vertices of the complex will be vertices of the graph. This should not lead to a confusion.

A simplicial cell complex~$Z$ is called an  $n$-dimensional \textit{simplicial cell pseudo-manifold} if every simplex of~$Z$ is contained in an $n$-dimensional simplex of~$Z$, and any $(n-1)$-dimensional simplex of~$Z$ is contained in exactly two  $n$-dimensional simplices of~$Z$. Equivalently, an  $n$-dimensional simplicial cell complex~$Z$ is a pseudo-manifold if and only if $Z\setminus\Sk^{n-2}(Z)$ is a  (non-compact) manifold without boundary. A pseudo-manifold~$Z$ is said to be \textit{oriented} if the $n$-dimensional simplices of it are endowed with compatible orientations. 

By definition, a \textit{topological subdivision\/} of a simplicial cell complex~$Z$ is a pair~$(Y,h)$ such that $Y$ is a simplicial cell complex and $h\colon Y\to Z$ is a piecewise linear homeomorphism such that the pre-image~$h^{-1}(\rho)$ of every simplex~$\rho$ of~$Z$ is a subcomplex of~$Y$. Similarly, a  \textit{topological subdivision\/} of a convex polytope~$Q$ is a pair~$(Y,h)$ such that $Y$ is a simplicial cell complex and $h\colon Y\to Q$ is a piecewise linear homeomorphism such that the pre-image~$h^{-1}(F)$ of every face~$F$ of~$Q$ is a subcomplex of~$Y$. The introduced concept of a topological subdivision should be distinguished from a more common and more restrictive concept of a \textit{geometric subdivision}. In the definition of a geometric subdivision simplices of~$Z$ (or the polytope~$Q$) are  required to be decomposed into true convex simplices rather than into their images under a piecewise linear homeomorphism.

For a polytope~$Q$, we denote by~$\Int Q$ the \textit{relative interior\/} of it, that is, the interior of~$Q$ in the affine hull of it. If $Q$ is a point, then $\Int Q=Q$.

\begin{propos}\label{propos_subd}
Suppose that $\Gamma$ is a connected graph on the vertex set $\{0,\ldots,n\}$. Then any $n$-dimensional simplicial cell pseudo-manifold~$Z$ has a topological subdivision~$(Y,h)$ with a regular colouring of vertices  $C\colon V(Y)\to\{0,\ldots,n\}$ satisfying the following condition:
\begin{itemize}
\item[$(*)$] Every $(n-2)$-dimensional simplex~$\rho$ of~$Y$ such that the two elements of the two-element set $\{0,\ldots,n\}\setminus C(\rho)$ are not connected by an edge in~$\Gamma$ is contained in exactly four $n$-dimensional simplices of~$Y$. 
\end{itemize}
\end{propos}

Notice that if $\Gamma=L_{n+1}$ is the path graph, then this proposition is obvious. Indeed, for the required subdivision one can take the barycentric subdivision of~$Z$ with the barycentre of every $i$-dimensional simplex of~$Z$ coloured in colour~$i$. The condition~$(*)$ can be checked easily. The proof of Proposition~\ref{propos_subd} in the general case will be based on the following auxiliary lemma.

\begin{lem}\label{lem_subd}
Suppose that\/ $\Gamma$ is a connected graph, $|V(\Gamma)|=n+1$, and  $a$ is a vertex of\/~$\Gamma$. Then there exists a topological subdivision~$(K,f)$ of the standard $n$-dimensional simplex~$\Delta^n$ with a regular colouring of vertices $C\colon V(K)\to V(\Gamma)$ satisfying the following two conditions:
\begin{enumerate}
\item $C(v)\ne a$ for all vertices $v$ on the boundary of the disk~$K$;
\item If  $\rho$ is an $(n-2)$-dimensional simplex of~$K$ such that the two elements of the two-element set $V(\Gamma)\setminus C(\rho)$ are not connected by an edge in~$\Gamma$,  then either $f(\Int\rho)\subset\Int\Delta^n$ and $\rho$ is contained in exactly four $n$-dimensional simplices of~$K$ or $f(\Int\rho)$ is contained in the relative interior of a facet of~$\Delta^n$  and $\rho$ is contained in exactly two $n$-dimensional simplices of~$K$.
\end{enumerate} 
\end{lem}

\begin{proof}
Let us prove the assertion of the lemma by induction on~$n$. The base of induction for $n=0$ is trivial, since the conditions~1 and~2 are void.
Assume that the assertion of the lemma is true for graphs with not more than $n$ vertices, and prove it for a graph~$\Gamma$ with  $n+1$ vertices. 

Consider the graph~$\Gamma-a$ obtained from~$\Gamma$ by removing the vertex~$a$ and all (open) edges connecting to it. Denote by $\Gamma_1,\ldots,\Gamma_s$ the connected components of $\Gamma-a$. Let $n_1+1,\ldots,n_s+1$ be the numbers of vertices of the graphs $\Gamma_1,\ldots,\Gamma_s$, respectively; then  $n=n_1+\cdots+n_s+s$. Since the graph~$\Gamma$ is connected, every connected component~$\Gamma_i$ contains at least one vertex that is connected by an edge with~$a$; we choose an arbitrary vertex satisfying this condition and denote it by~$b_i$. By the induction hypothesis, the assertion of the lemma is true for all pairs~$(\Gamma_i,b_i)$. This means that, for every~$i=1,\ldots,s$, there exists a topological subdivision~$(K_i,f_i)$ of the standard simplex~$\Delta^{n_i}$ with a regular colouring of vertices $C_i\colon V(K_i)\to V(\Gamma_i)$ that satisfies conditions~1 (with $a$ replaced by~$b_i$) and~2.

For every~$i$, consider the Euclidean space~$\R^{n_i+1}$ with the standard orthonormal basis $e_0,\ldots,e_{n_i}$ and the standard simplex~$\Delta^{n_i}\subset\R^{n_i+1}$ with vertices $e_0,\ldots,e_{n_i}$. Consider the group $G_i\cong\Z_2^{n_i+1}$ of isometries of~$\R^{n_i+1}$ generated by the orthogonal reflections in the coordinate hyperplanes $\{x_0=0\},\ldots,\{x_{n_i}=0\}$, where $x_0,\ldots,x_{n_i}$ are the coordinates in the basis $e_0,\ldots,e_{n_i}$. The simplices~$g(\Delta^{n_i})$, $g\in G_i$, constitute the boundary of the  \textit{cross-polytope}~$Q^{n_i+1}$, which is the regular convex polytope in~$\R^{n_i+1}$ with vertices $\pm e_0,\ldots,\pm e_{n_i}$. Subdivide every simplex~$g(\Delta^{n_i})$ by means of the topological subdivision~$(K_i,g\circ f_i)$. Then we obtain the topological subdivision of the boundary of~$Q^{n_i+1}$; we denote this subdivision by~$(\bK_i,\bar f_i)$. The regular colouring~$C_i$ of vertices of~$K_i$ induces the regular colouring~$\overline{C}_i$ of vertices of~$\bK_i$ in colours in the same set~$V(\Gamma_i)$.  Condition~2 for~$K_i$ provides that any  $(n_i-2)$-dimensional simplex~$\rho$ of~$\bK_i$ is contained in exactly four  $n_i$-dimensional simplices of~$\bK_i$ whenever the two elements of the two-element set~$V(\Gamma_i)\setminus \overline{C}_i(\rho)$ are not connected by an edge.

We identify naturally the boundary~$\partial Q^n$ of the cross-polytope $Q^{n}\subset\R^n$ with the join $\partial Q^{n_1+1} * \cdots *\partial Q^{n_s+1}$, and consider the topological subdivision~$(J,q)$ of~$\partial Q^n$  that is obtained by taking the join of the topological subdivisions $(\bK_1,\bar f_1),\ldots,(\bK_s,\bar f_s)$. Then vertices of~$J$ are regularly coloured in colours in the set $V(\Gamma_1)\cup\cdots\cup V(\Gamma_s)=V(\Gamma)\setminus\{a\}$.

The cross-polytope~$Q^n$ is naturally identified with the cone over~$\partial Q^n$ with apex in the origin of~$\R^n$.  We put $K=\cone(J)$, $\tilde q=\cone(q)$. Then $(K,\tilde q)$ is a topological subdivision of~$Q^n$. We denote the apex of the cone~$K=\cone(J)$ by~$p$ and colour it in the colour~$a$. Then we obtain the regular colouring~$C$ of vertices of~$K$ in colours in the set~$V(\Gamma)$.

Now, consider an arbitrary piecewise linear homeomorphism $\varphi\colon Q^n\to\Delta^n$ such that the pre-image of every face of~$\Delta^n$ is the union of several closed faces of~$Q^n$. For instance, such homeomorphism can be constructed in the following way. Denote the vertices of~$Q^n$ by $
\pm \varepsilon_1,\ldots,\pm \varepsilon_n$ and the centre of~$Q^n$ by~$o$. (We use notation~$\pm\varepsilon_i$ instead of~$\pm e_i$ to avoid confusion between vertices of the cross-polytope~$Q^n$ and vertices of the standard simplex~$\Delta^n$.) For a triangulation of the cross-polytope~$Q^n$, we take the cone with apex~$o$ over the natural triangulation of~$\partial Q^n$. Define a mapping~$\varphi$ on the vertices of this triangulation by
\begin{gather*}
\varphi(\varepsilon_i)=e_i,\qquad
\varphi(-\varepsilon_i)=\frac{1}{i}(e_0+\cdots+e_{i-1}),\qquad i=1,\ldots,n,\\
\varphi(o)=\frac{1}{n+1}(e_0+\cdots+e_n),
\end{gather*} 
and extend it linearly to every simplex of the triangulation. It is easy to check that the obtained mapping is a piecewise linear homeomorphism and satisfies the required condition that the pre-images of faces of~$\Delta^n$ are unions of faces of the cross-polytope. Then the mapping 
$
f=\varphi\circ\tilde q\colon K\to\Delta^n
$
provides a topological subdivision of the simplex~$\Delta^n$.

Obviously, condition~1 holds for the subdivision~$(K,f)$. Let us prove condition~2. Consider an arbitrary $(n-2)$-dimensional simplex~$\rho$ of~$K$. Suppose that $V(\Gamma)\setminus C(\rho)=\{c_1,c_2\}$, and $c_1$ and~$c_2$ are not connected by an edge in~$\Gamma$. Consider three cases:

\textsl{1. The vertices $c_1$ and $c_2$ lie in distinct components\/~$\Gamma_{j_1}$ and\/~$\Gamma_{j_2}$ of\/~$\Gamma-a$.} Then $a\in C(\rho)$, hence,  $p\in \rho$. Therefore, $f(\Int\rho)\subset\Int\Delta^n$. Denote by~$\tau$ the facet of~$\rho$ opposite to the vertex~$p$.  We have, $\tau=\tau_1*\cdots*\tau_s$, where $\tau_1,\ldots,\tau_s$ are simplices of the complexes $\bK_1,\ldots,\bK_s$, respectively, such that  $\dim\tau_i=n_i$ whenever $i\ne j_1,j_2$,  $\dim\tau_{j_1}=n_{j_1}-1$, and $\dim\tau_{j_2}=n_{j_2}-1$. Since every complex~$\bK_i$ is an $n_i$-dimensional simplicial cell sphere, hence, an $n_i$-dimensional pseudo-manifold, it follows that~$\tau$ is contained in exactly four  $(n-1)$-dimensional simplices of the complex $J=\bK_1*\cdots*\bK_s$.  
Therefore, the simplex~$\rho=p*\tau$ is contained in exactly four $n$-dimensional simplices of the complex~$K=\cone(J)$.

\textsl{2. The vertices $c_1$ and $c_2$ lie in the same component~$\Gamma_j$ of\/~$\Gamma-a$.} As in the previous case, we have  $p\in \rho$, hence, $f(\Int\rho)\subset\Int\Delta^n$. The facet $\tau\subset\rho$ opposite to~$p$ has the decomposition $\tau=\tau_1*\cdots*\tau_s$, where $\tau_1,\ldots,\tau_s$ are simplices of the complexes $\bK_1,\ldots,\bK_s$, respectively, such that $\dim\tau_i=n_i$ unless $i= j$, and $\dim\tau_{j}=n_{j}-2$. As we have already mentioned, condition~2 for~$K_j$ implies that $\tau_j$ is contained in exactly four $n_j$-dimensional simplices of~$\bK_j$. 
Hence, the simplex~$\rho=p*\tau_1*\cdots*\tau_s$ is contained in exactly four $n$-dimensional simplices of the complex $K=p*\bK_1*\cdots*\bK_s$.

\textsl{3. One of the vertices~$c_1$ and~$c_2$  \textnormal{(}say,~$c_1$\textnormal{)} coincides with~$a$.} Then $a\notin C(\rho)$, hence, $\rho$ is contained in $\partial K=J$. Since~$J$ is homeomorphic to the $(n-1)$-dimensional sphere and, in particular, is an $(n-1)$-dimensional  pseudo-manifold, we obtain that $\rho$ is contained in exactly two $(n-1)$-dimensional simplices of~$J$. Therefore, $\rho$ is contained in exactly two $n$-dimensional simplices of~$K=\cone(J)$. Let us prove that the set $f(\Int\rho)$ is contained in the relative interior of a facet of~$\Delta^n$. Since we already know that $f(\rho)\subset\partial\Delta^n$ we suffice to prove that  $f(\rho)$ is contained in no $(n-2)$-dimensional face of~$\Delta^n$. Further, to prove this we suffice to prove that the set $\tilde q(\rho)$ is contained in no $(n-2)$-dimensional face of~$Q^n$. Suppose that $c_2\in V(\Gamma_j)$. Then  $\rho=\rho_1*\cdots*\rho_s$, where $\rho_1,\ldots,\rho_s$ are simplices of the complexes $\bK_1,\ldots,\bK_s$, respectively, such that $\dim\rho_i=n_i$ unless $i= j$, and $\dim\rho_j=n_j-1$. Then  $q(\rho)=\bar f_1(\rho_1)*\cdots*\bar f_s(\rho_s)$. Since the subdivision~$(J,q)$ of~$\partial Q^n$ is invariant under the action of the group $G=G_1\times\cdots\times G_s$, we may assume without loss of generality that the simplices $\rho_1,\ldots,\rho_s$ are contained in the subcomplexes $K_1,\ldots,K_s$ of the complexes $\bK_1,\ldots,\bK_s$, respectively.  Then, for every $i\ne j$, the set~$f_i(\rho_i)$ contains a point in the relative interior of~$\Delta^{n_i}$. Since the vertices $c_1=a$ and~$c_2$ are not connected by an edge, we have $c_2\ne b_j$. Hence, $b_j\in C_j(\rho_j)$. Therefore, condition~1 for the topological subdivision~$(K_j,f_j)$ implies that the simplex~$\rho_j$ contains a vertex that does not lie on the boundary of the disk~$K_j$. Consequently, the set~$f_j(\rho_j)$ also contains a point in the relative interior of~$\Delta^{n_j}$. Thus, the set~$q(\rho)$ contains a point in the relative interior of the facet~$\Delta^{n-1}=\Delta^{n_1}*\cdots *\Delta^{n_s}$  of~$Q^n$.
\end{proof}

\begin{proof}[of Proposition~\ref{propos_subd}] Replacing, if necessary, the pseudo-manifold~$Z$ with its barycntric subdivision, we may assume that the vertices of~$Z$ admit a regular colouring in colours in the set~$\{0,\ldots,n\}$. For each  $n$-dimensional simplex~$\sigma$ of~$Z$, consider the affine isomorphism $\psi_{\sigma}\colon\Delta^n\to\sigma$ taking the vertices $e_0,\ldots,e_n$ of~$\Delta^n$ to the vertices of~$\sigma$ of colours~$0,\ldots,n$, respectively. Let  $f\colon K\to \Delta^n$ be the topological subdivision in Lemma~\ref{lem_subd} for the graph~$\Gamma$ and a vertex~$a$ of it. Subdivide every $n$-dimensional simplex~$\sigma$ of~$Z$ by means of the subdivision~$(K,\psi_{\sigma}\circ f)$. Denote the obtained subdivision of~$Z$ by~$(Y,h)$. Condition~$(*)$ for the subdivision~$(Y,h)$ follows immediately from condition~2 in Lemma~\ref{lem_subd} for the subdivision~$(K,f)$.
\end{proof}

\section{A combinatorial construction}\label{section_constr}

In this section we describe an auxiliary combinatorial construction that will be used in the next section in the proof of Theorem~\ref{theorem_main}. This construction generalizes the construction suggested by the author in~\cite{Gai07} and developed in~\cite{Gai08a}, \cite{Gai08b}, \cite{Gai08c}, and~\cite{Gai13}.

Let $\Gamma$ be a connected graph on the vertex set $\{0,\ldots,n\}$. Put $m=|\CB'(\Gamma)|$. Then $m$ is the number of facets of the graph-associahedron~$P_{\Gamma}$. As above, we conveniently index elements of the standard basis of~$\Z_2^m$ by subsets~$S\in\CB'(\Gamma)$, and denote the basis element corresponding to a subset~$S$ by~$a_S$.

Let $\Sigma=\Sigma_+\sqcup \Sigma_-$ be a finite set with involutions $\xi_0,\ldots,\xi_n$ such that every~$\xi_i$ exchanges the subsets~$\Sigma_+$ and~$\Sigma_-$. Assume that the involutions~$\xi_i$ and~$\xi_j$ commute for any vertices~$i$ and~$j$ that are not connected by an edge in~$\Gamma$. 

Suppose that $S\in\CB'(\Gamma)$.  Denote by~$\min(S)$ the least element of the set~$S$. Let~$\CI_{S}$ be the set consisting of all involutions $\mu\colon \Sigma\to \Sigma$ such that
$$
\mu=\xi_{i_1}\circ\cdots\circ\xi_{i_q}\circ\xi_{\min(S)}\circ\xi_{i_q}\circ\cdots\circ\xi_{i_1}
$$
for some elements $i_1,\ldots,i_q\in S$, which are not required to be distinct. Since every involution~$\xi_i$ exchanges the subsets~$\Sigma_+$ and~$\Sigma_-$, the same holds true for any involution~$\mu\in \CI_{S}$.

It is easy to see that if $S_1,S_2\in\CB'(\Gamma)$ and $S_1\subseteq S_2$, then the involution $\mu_1\circ\mu_2\circ\mu_1$ belongs to~$\CI_{S_2}$ for any~$\mu_1\in \CI_{S_1}$ and~$\mu_2\in\CI_{S_2}$. 

Now, suppose that $S_1,S_2\in\CB'(\Gamma)$ and $S_1\cup S_2\notin\CB(\Gamma)$. Then $S_1\cap S_2=\varnothing$ and no element of~$S_1$ is connected by an edge in~$\Gamma$ with any element of~$S_2$. Hence the involutions~$\xi_{i_1}$ and~$\xi_{i_2}$ commute for any $i_1\in S_1$ and~$i_2\in S_2$. Therefore, any involutions~$\mu_1\in\CI_{S_1}$ and~$\mu_2\in\CI_{S_2}$ commute.

We put,
\begin{equation}\label{eq_Omega}
\Omega_{\pm}=\Sigma_{\pm}\times\left(\prod_{T\in\CB'(\Gamma)}\CI_{T}\right)\times\Z_2^{m},\qquad \Omega=\Omega_+\sqcup\Omega_-\,.
\end{equation}
An arbitrary element of~$\Omega$ has the form $\bigl(\sigma,(\mu_{T})_{T\in\CB'(\Gamma)},g\bigr)$, where $\sigma\in \Sigma$, $\mu_{T}\in\CI_{T}$ for all~$T\in\CB'(\Gamma)$, and $g\in \Z_2^{m}$.

We define the mappings $\varphi_{S}\colon \Omega\to \Omega$, $S\in\CB'(\Gamma)$, by
\begin{gather}\label{eq_phi}
\varphi_{S}\bigl(\sigma,(\mu_{T})_{T\in\CB'(\Gamma)},g\bigr)=
\bigl(\mu_{S}(\sigma),(\widetilde\mu_{T})_{T\in\CB'(\Gamma)},g+a_{S}\bigr),\\
\widetilde\mu_{T}=\left\{
\begin{aligned}
&\mu_{S}\circ\mu_{T}\circ\mu_{S}&&\text{if $S\subseteq T$,}\\
&\mu_{T}&&\text{if $S\not\subseteq T$.}
\end{aligned}
\right.\nonumber
\end{gather}
It has already been mentioned above that if $S\subseteq T$ then the involution $\mu_{S}\circ\mu_{T}\circ\mu_{S}$ belongs to~$\CI_{T}$. Hence, $\widetilde\mu_{T}\in\CI_{T}$ for all~$T\in\CB'(\Gamma)$. Since $\widetilde\mu_S=\mu_S$, we obtain that $\varphi_S^2=\mathrm{id}_{\Omega}$, that is, $\varphi_S$ is an involution. The involution~$\mu_S$ exchanges the subsets~$\Sigma_+$ and~$\Sigma_-$. Therefore, the involution~$\varphi_{S}$ exchanges the subsets~$\Omega_+$ and~$\Omega_-$.

\begin{propos}\label{propos_phi_commute}
The involutions~$\varphi_{S_1}$ and~$\varphi_{S_2}$ commute for any $S_1,S_2\in \CB'(\Gamma)$ such that the facets~$F_{S_1}$ and~$F_{S_2}$ of~$P_{\Gamma}$ intersect each other.
\end{propos}

\begin{proof}
The facets~$F_{S_1}$ and~$F_{S_2}$ intersect each other if and only if either one of the two subsets~$S_1$ and~$S_2$ is contained in the other or $S_1\cup S_2\notin\CB(\Gamma)$. In the first case, we may assume that $S_1\subseteq S_2$. A direct computation using~\eqref{eq_phi} yields
\begin{gather*}
\begin{split}
(\varphi_{S_1}\circ\varphi_{S_2})\bigl(\sigma,(\mu_{T})_{T\in\CB'(\Gamma)},g\bigr)=(\varphi_{S_2}\circ\varphi_{S_1})\bigl(\sigma,(\mu_{T})_{T\in\CB'(\Gamma)},g\bigr)={} \\
{}=\bigl((\mu_{S_1}\circ\mu_{S_2})(\sigma),(\widehat\mu_{T})_{T\in\CB'(\Gamma)},g+a_{S_1}+a_{S_2}\bigr),
\end{split}\\
\widehat\mu_{T}=\left\{
\begin{aligned}
&\mu_{S_1}\circ\mu_{S_2}\circ\mu_{T}\circ\mu_{S_2}\circ\mu_{S_1}&&\text{if $S_2\subseteq T$,}\\
&\mu_{S_1}\circ\mu_{T}\circ\mu_{S_1}&&\text{if $S_1\subseteq T$ and $S_2\not\subseteq T$,}\\
&\mu_{T},&&\text{if $S_1\not\subseteq T$.}
\end{aligned}
\right.\nonumber
\end{gather*}
In the second case, we have $S_1\cup S_2\notin\CB(\Gamma)$. Hence any involution in~$\CI_{S_1}$ commutes with any involution in~$\CI_{S_2}$, which immediately implies that $\varphi_{S_1}\circ\varphi_{S_2}=\varphi_{S_2}\circ\varphi_{S_1}$.
\end{proof}

\section{Proof of Theorem~\ref{theorem_main}}\label{section_proof}

It follows easily from the definition of singular homology that any homology class~$z\in H_n(X;\Z)$ of any topological space~$X$ can be realized as a continuous image of the fundamental homology class of an oriented $n$-dimensional simplicial cell pseudo-manifold~$Z$, i.\,e., there exists a continuous mapping $\alpha\colon Z\to X$ such that $\alpha_*[Z]=z$. By Proposition~\ref{propos_subd}, there exists a topological subdivision~$(Y,h)$ of~$Z$ with a regular colouring of vertices $C\colon V(Y)\to\{0,\ldots,n\}$ satisfying condition~$(*)$. Then $(\alpha\circ h)_*[Y]=z$. 

We denote by~$\Sigma$ the set of $n$-dimensional simplices of~$Y$. For each simplex~$\sigma\in\Sigma$, consider the affine isomorphism $\psi_{\sigma}\colon\Delta^n\to\sigma$ taking the vertices $e_0,\ldots,e_n$ of the standard simplex~$\Delta^n$ to the vertices of~$\sigma$ of colours~$0,\ldots,n$, respectively. We denote by~$\Sigma_+$ the set of all simplices $\sigma\in\Sigma$ such that the isomorphism~$\psi_{\sigma}$ preserves the orientation, and we denote by~$\Sigma_-$ the set of all simplices~$\sigma\in\Sigma$ such that the isomorphism~$\psi_{\sigma}$ reverses the orientation.
Obviously, if two different simplices $\sigma_1,\sigma_2\in\Sigma$ have a common $(n-1)$-dimensional face, then one of these two simplices belongs to~$\Sigma_+$, and the other belongs to~$\Sigma_-$.

For each $i\in\{0,\ldots,n\}$ and each~$\sigma\in\Sigma$, we denote by~$\xi_i(\sigma)$ a unique simplex in~$\Sigma$ such that $\xi_i(\sigma)\ne\sigma$ and the simplices~$\sigma$ and~$\xi_i(\sigma)$ have a common facet~$\tau$ whose set of colours of vertices is
$$C(\tau)=\{0,\ldots,n\}\setminus\{i\}.$$
(Since $Y$ is not necessarily a simplicial complex but only a simplicial cell complex, it is possible that $\tau$ is not a unique common facet of the simplices~$\sigma$ and~$\xi_i(\sigma)$.)

Then $\xi_i\colon\Sigma\to\Sigma$ are involutions satisfying $\xi_i(\Sigma_+)=\Sigma_-$ and $\xi_i(\Sigma_-)=\Sigma_+$. By condition~$(*)$ in Proposition~\ref{propos_subd}, each $(n-2)$-dimensional simplex~$\rho$ of~$Y$ such that the two elements of the two-element set $\{i,j\}=\{0,\ldots,n\}\setminus C(\rho)$ are not connected by an edge in~$\Gamma$ is contained in exactly four  $n$-dimensional simplices. It is easy to see that if we denote by~$\sigma$ one of these four simplices, then the three others will be $\xi_i(\sigma)$, $\xi_j(\sigma)$, and $\xi_i(\xi_j(\sigma))=\xi_j(\xi_i(\sigma))$. Thus, the involutions~$\xi_i$ and~$\xi_j$ commute whenever $i$ and~$j$ are not connected by an edge in~$\Gamma$.

We apply the construction in the previous section to the set~$\Sigma$ with the involutions~$\xi_0,\ldots,\xi_n$. This means that, we introduce the set $\Omega=\Omega_+\sqcup\Omega_-$ and the involutions $\varphi_S\colon\Omega\to\Omega$, $S\in\CB'(\Gamma)$ by~\eqref{eq_Omega},~\eqref{eq_phi}. 

We denote by~$\Phi$ the subgroup of the symmetric group on~$\Omega$ generated by the involutions~$\varphi_S$, $S\in\CB'(\Gamma)$. The action of the involutions~$\varphi_S$ on the factor~$\Z_2^m$ in decomposition~\eqref{eq_Omega} yields a well-defined epimorphism $\varkappa\colon\Phi\to\Z_2^m$ such that $\varkappa(\varphi_S)=a_S$ for all~$S$.

For each point $x\in P_{\Gamma}$, we denote by~$\Phi(x)$ the subgroup of~$\Phi$ generated by all~$\varphi_S$ such that~$x\in F_S$. Suppose that~$x$ lies in the relative interior of an $(n-k)$-dimensional face $F_{S_1}\cap\cdots\cap F_{S_k}$ of~$P_{\Gamma}$. By Proposition~\ref{propos_phi_commute}, the involutions $\varphi_{S_1},\ldots,\varphi_{S_k}$ pairwise commute. Hence the restriction of~$\varkappa$ to~$\Phi(x)$ is the isomorphism onto  the subgroup $\Z_2^k(x)\subset\Z_2^m$ generated by the elements~$a_{S_1},\ldots,a_{S_k}$.

We put,
\begin{equation*}
N=(P_{\Gamma}\times\Omega)/\sim\,,
\end{equation*}
where $\sim$ is the equivalence relation  on $P_{\Gamma}\times\Omega$ such that $(x,\omega)\sim (x',\omega')$ if and only if $x=x'$ and $\theta(\omega)=\omega'$ for some~$\theta\in\Phi(x)$. We denote by~$[x,\omega]$ the point in~$N$ corresponding to the equivalence class of the pair~$(x,\omega)$. Consider the mapping $p\colon N\to\CR_{P_{\Gamma}}$ given by
\begin{equation*}
p([x,\omega])=[x,\pr_3(\omega)],
\end{equation*}
where $\pr_3$ is the projection onto the third factor in decomposition~\eqref{eq_Omega}.

\begin{propos}
The mapping $p$ is an $r$-fold covering, where $$r=|\Sigma|\prod_{T\in\CB'(\Gamma)}|\CI_T|.$$
\end{propos}

\begin{proof}
Since the spaces~$N$ and~$\CR_{P_{\Gamma}}$ are compact and~$\CR_{P_{\Gamma}}$ is connected, to prove that  $p$ is a covering, it is sufficient to prove that $p$ is a local homeomorphism, that is, to prove that for each point~$[x,\omega]\in N$ the restriction of $p$ to a neighborhood of~$[x,\omega]$ is a homeomorphism onto a neighborhood of~$p([x,\omega])$. Suppose that the point~$x$ lies in the relative interior of an $(n-k)$-dimensional face $F_{S_1}\cap\cdots\cap F_{S_k}$ of~$P_{\Gamma}$, and $\omega=(\sigma,(\mu_{T})_{T\in\CB'(\Gamma)},g)$. Then $p([x,\omega])=[x,g]$. Every point in a neighborhood of~$[x,g]$ has the form $[y,g+h]$, where $y$ is close to~$x$ and $h\in\Z_2^k(x)$. 
A  locally inverse to~$p$ mapping~$q$ is given by $q([y,g+h])=[y,\nu(h)(\omega)]$, where $\nu\colon\Z_2^k(x)\to\Phi(x)$ is the isomorphism inverse to~$\varkappa|_{\Phi(x)}$. It can be immediately checked that the mapping~$q$ is well defined and continuous in a neighborhood of~$[x,g]$, and is both the left and the right inverse to~$p$. Therefore, $p$ is a local homeomorphism, hence, a covering.

The manifold~$\CR_{P_{\Gamma}}$ is glued out of $2^{m}$ copies of~$P_{\Gamma}$, and the manifold~$N$ is glued out of~$2^{m}r$ copies of~$P_{\Gamma}$. The mapping~$p$ maps every copy of~$P_{\Gamma}$ in the decomposition of~$N$ isomorphically onto a copy of~$P_{\Gamma}$ in the decomposition of~$\CR_P$. Hence, the number of sheets of the covering~$p$ is equal to~$r$.
\end{proof}

For an $\omega=(\sigma,(\mu_T)_{T\in\CB'(\Gamma)},g)$, we put~$\varepsilon'(\omega)=1$ if $\sigma\in\Sigma_+$ and $\varepsilon'(\omega)=-1$ if $\sigma\in\Sigma_-$. Consider the homomorphism $\eta\colon\Z_2^{m}\to\Z_2$ that takes every~$a_S$ to the generator~$1$ of the group~$\Z_2=\Z/2\Z$, and put $\varepsilon''(\omega)=(-1)^{\eta(g)}$. Further, we put $\varepsilon(\omega)=\varepsilon'(\omega)\varepsilon''(\omega)$. It follows from~\eqref{eq_phi} that $\varepsilon(\varphi_S(\omega))=\varepsilon(\omega)$ for all~$S$ and~$\omega$. Hence, $\varepsilon(\omega_1)=\varepsilon(\omega_2)$ whenever $[x,\omega_1]=[x,\omega_2]$. Therefore, the manifold~$N$ is the disjoint union of the manifold~$N_+$ consisting of all points~$[x,\omega]$ such that $\varepsilon(\omega)=1$ and the manifold~$N_-$ consisting of all points~$[x,\omega]$ such that $\varepsilon(\omega)=-1$. Either of the manifolds~$N_+$ and~$N_-$ is an $(r/2)$-fold covering of~$\CR_{P_{\Gamma}}$.

Consider the mapping $\gamma\colon N_+\to Y$ given by
\begin{equation}\label{eq_gamma}
\gamma([x,(\sigma,(\mu_T)_{T\in\CB'(\Gamma)},g)])=\psi_{\sigma}(\pi_{\Gamma}(x)),
\end{equation}
where $\pi_{\Gamma}\colon P_{\Gamma}\to\Delta^n$ is the mapping constructed in Section~\ref{subsection_as}.

\begin{propos}
The mapping $\gamma$ is well defined, continuous, and 
\begin{equation*}
\gamma_*[N_+]=s[Y],\qquad s=2^{m-1}\prod_{T\in\CB'(\Gamma)}|\CI_T|.
\end{equation*}
\end{propos}

\begin{proof}
To prove that~$\gamma$ is well defined and continuous, we need to show that the values~$\gamma([x,\omega])$ and~$\gamma([x,\omega'])$ computed by~\eqref{eq_gamma} are equal to each other whenever $(x,\omega)\sim (x,\omega')$. To this end, it is sufficient to show that the values~$\gamma([x,\omega])$ and~$\gamma([x,\varphi_S(\omega)])$ are  equal to each other whenever  $x\in F_S$. Suppose that $\omega=(\sigma,(\mu_T)_{T\in\CB'(\Gamma)},g)$; then $\varphi_S(\omega)=(\mu_S(\sigma),(\widetilde{\mu}_T)_{T\in\CB'(\Gamma)},g+a_S)$. It follows from assertion~1 of Proposition~\ref{propos_pi} that $\pi_{\Gamma}(x)\in\Delta_{V\setminus S}$ whenever  $x\in F_S$. Hence, $\psi_{\xi_i(\tau)}(\pi_{\Gamma}(x))=\psi_{\tau}(\pi_{\Gamma}(x))$ for all~$\tau\in\Sigma$ and all~$i\in S$. Therefore, $\psi_{\mu_S(\sigma)}(\pi_{\Gamma}(x))=\psi_{\sigma}(\pi_{\Gamma}(x))$, which is exactly what we need to prove.

We choose the orientation of~$\CR_{P_{\Gamma}}$ so that the embedding of~$P_{\Gamma}$ into~$\CR_{P_{\Gamma}}$ given by~$x\mapsto[x,g]$ preserves the orientation if and only if $\eta(g)=0$. As before, we endow the covering~$N_+$ of~$\CR_{P_{\Gamma}}$ with the induced orientation. The embedding of~$P_{\Gamma}$ into~$N_+$ given by~$x\mapsto[x,\omega]$ preserves the orientation if and only if $\varepsilon''(\omega)=1$. Since $\varepsilon(\omega)=1$, the latter is equivalent to~$\varepsilon'(\omega)=1$. The embedding $\psi_{\sigma}\colon\Delta^n\to Y$ also preserves the orientation if and only if~$\varepsilon'(\omega)=1$. By assertion~3 of Proposition~\ref{propos_pi}, the mapping~$\pi_{\Gamma}$ has degree~$1$. Therefore, $\gamma$ maps every cell of~$N_+$ isomorphic to~$P_{\Gamma}$ onto a simplex of~$Y$ with degree~$1$. Besides, for each simplex of~$Y$, there are exactly~$s$ cells of~$N_+$ that are mapped onto it. Consequently, $\gamma_*[N_+]=s[Y]$.
\end{proof}

Thus, starting from an arbitrary homology class~$z\in H_n(X;\Z)$, we have constructed an $(r/2)$-fold covering~$N_+$ of~$\CR_{P_{\Gamma}}$ and a mapping $f=\alpha\circ h \circ\gamma$ of~$N_+$ to~$X$ such that $f_*[N_+]=sz$ for certain~$s>0$. Hence, $\CR_{P_{\Gamma}}$ is a URC-manifold.

\begin{remark}
We could define the mapping~$\gamma$ of the whole manifold~$N$ onto~$Y$ again by~\eqref{eq_gamma}. Nevertheless, this mapping would have zero degree, since the $n$-dimensional cells of~$N_+$ would be mapped onto simplices of~$Y$ preserving the orientation, and the $n$-dimensional cells of~$N_-$ would be mapped onto simplices of~$Y$ reversing the orientation. 
\end{remark}

\section{Looking for the smallest URC-manifold}\label{section_small}

Since for $n\ge 3$ the class of $n$-dimensional URC-manifolds is rather extensive, an interesting problem is to find an $n$-dimensional URC-manifold that is in some sense the smallest in this class. In the two-dimensional case, URC-manifolds are exactly oriented surfaces of genera $g\ge 2$. Hence, in any reasonable sense the smallest of them is the surface of genus~$2$. In high-dimensional case, the situation is more complicated, and the answer to the question on the smallest URC-manifold depends on how to compare different manifolds.

The problem on realization of cycles by images of spheres is the classical problem on the image of the Hurewicz homomorphism. It is well known that the Hurewicz homomorphism $\pi_n(X)\otimes \Q\to H_n(X;\Q)$ for $n\ge 2$ is not always surjective. Hence, the sphere~$S^n$ is by no means a URC-manifold. However, notice that the classical results by Serre~\cite{Ser51} imply that after taking  a multiple suspension the Hurewicz homomorphism $\pi_{n+N}(\Sigma^NX)\otimes \Q\to H_{n+N}(\Sigma^NX;\Q)\cong H_n(X;\Q)$ becomes an isomorphism. Therefore, any homology class can be stably realized with some multiplicity by an image of the fundamental class of the sphere. Upon this fact is based the construction of the Chern--Dold character in any extraodinary cohomology theory due to Buchstaber~\cite{Buc70-2}. In the present paper we always consider the question on unstable realization of cycles.

It is easy to see that if a mapping $f\colon M^n\to N^n$ has nonzero degree, then the image of the homomorphism $f_*\colon \pi_1(M^n)\to \pi_1(N^n)$ has finite index in~$\pi_1(N^n)$, and the homomorphism $f_*\colon H_*(M^n;\Q)\to H_*(N^n;\Q)$ is surjective. It follows easily that no manifold with finite fundamental group is a URC-manifold. Moreover, any URC-manifold must have connected finite-fold coverings with arbitrarily large Betti numbers. It follows from a result by Kotschick and L\"oh~\cite{KoLo09} that no direct product of two manifolds of positive dimensions is a URC-manifold.

In dimension~$3$, we know about the domination relation more than in higher dimensions. In particular, Sun~\cite{Sun15} has recently shown that any oriented hyperbolic manifold is a URC-manifold. (A manifold is called \textit{hyperbolic\/} if it admits a Riemannian metric of constant negative curvature.) There are many examples of three-dimensional homology spheres that are hyperbolic manifolds. By the result of Sun all they are URC-manifolds.

\begin{quest}\label{quest_homsphere}
For which $n\ge 4$ there exist $n$-dimensional homology spheres \textnormal{(}or at least rational homology spheres\textnormal{)} that are URC-manifolds?
\end{quest}

\begin{remark}
It has already been mentioned above that any nonzero degree mapping  $f\colon M^n\to N^n$ induces a surjective homomorphism in homology with rational coefficients. Hence no rational homology sphere can dominate a manifold that is not a rational homology sphere. However, this is not prevent a homology sphere from being a URC-manifold, since coverings of homology spheres can have nonzero (and arbitrarily large) Betti numbers.
\end{remark}

Let $\K$ be a field. Recall that the \textit{Betti numbers with coefficients in~$\K$} of a space~$X$ are the numbers $$\beta_i^{\K}(X)=\dim_{\K}H_i(X;\K),$$ and the  \textit{total Betti number with coefficients in~$\K$} of~$X$ is the sum $$\beta^{\K}(X)=\sum\beta_i^{\K}(X).$$ 
If $\K=\Q$ is the field of rational numbers, then for spaces with finitely generated homology groups, the numbers~$\beta_i^{\Q}$ are the usual Betti numbers, i.\,e., the ranks of the free parts of the groups~$H_i(X;\Z)$, respectively. 
Since the answer to Question~\ref{quest_homsphere} is unknown, it is reasonable to consider the following more general problem.

\begin{problem}
For every dimension $n\ge 4$, find an $n$-dimensional URC-\allowbreak manifold~$M^n$ with the smallest total Betti number~$\beta^{\K}(M^n)$. Does there exist a URC-manifold $M^n$ such that for any other URC-manifold~$N^n$ of the same dimension, the inequalities $\beta_i^{\K}(N^n)\ge \beta_i^{\K}(M^n)$ hold true for all~$i$?
\end{problem}

In light of this problem it would be interesting to find out which of the already constructed URC-manifolds has the smallest total Betti number. Thus, the following natural question arises. 

\begin{quest}\label{quest_sc}
Let $\CC_n$ be the class consisting of all orientable manifolds~$M_{P_{\Gamma},\lambda}$ and all two-fold orientation coverings~$\overline{M}_{P_{\Gamma},\lambda}$ of non-orientable manifolds~$M_{P_{\Gamma},\lambda}$, where $\Gamma$ runs over connected graphs on the vertex set $\{0,\ldots,n\}$. Which of the manifolds in~$\CC_n$ has the smallest total Betti number with coefficients in~$\K$? 
\end{quest}

We shall be interested in the cases $\K=\Q$ and $\K=\Z_2$. In light of the studied in the present paper problem on realization of cycles with multiplicities the most natural characteristics of URC-manifolds are Betti numbers with coefficients in~$\Q$. On the other hand, Betti numbers with coefficients in~$\Z_2$ are easier to compute for small covers. 

The author does not know the answer to Question~\ref{quest_sc}. Nevertheless, we shall at least show that among URC-manifolds constructed in the present paper there are manifolds whose total Betti numbers are much smaller than the total Betti numbers of small covers of permutohedra, which have been known to be  URC-manifolds before.  In particular, for such manifolds we can take the manifolds~$\overline{M}_{\As^n,\clambda}$. This support our intuition that the small covers of~$\As^n$ must be the `smallest' among the small covers of graph-associahedra corresponding to connected graphs, and must be `much smaller' then the small covers of~$\Pe^n$, at least for~$n$ large enough. This intuition takes its origin from a theorem of Buchstaber and Volodin~\cite{BuVo11} claiming that the numbers of faces of graph-associahedra corresponding to connected graphs are always greater than or equal to the corresponding numbers of faces of Stasheff associahedra~$\As^n$, and the same is true for other important characteristics such as $h$-numbers and $\gamma$-numbers. Moreover, for large~$n$, the  numbers of faces of~$\As^n$ are much smaller than the numbers of faces of~$\Pe^n$. For instance, the number of facets of~$\As^n$ is equal to $n(n+3)/2$ while the number of facets of~$\Pe^n$ is equal to~$2^{n+1}-2$.  

In Section~\ref{subsection_simp_vol} we shall also consider simplicial volume as another important characteristic that allows us to compare URC-manifolds. 

\subsection{Betti numbers with coefficients in~$\Q$} \label{subsection_rational}

Notice that if  $\overline{M}^n$ is a two-fold orientation covering of a non-orientable closed manifold~$M^n$, then by a classical result due to Eckmann~\cite{Eck49} we have $\beta^{\Q}(\overline{M}^n)=2\beta^{\Q}(M^n)$. Moreover, 
\begin{equation}\label{eq_Brash} 
\beta_i^{\Q}(\overline{M}^n)=\beta_i^{\Q}(M^n)+\beta_{n-i}^{\Q}(M^n)
\end{equation}
for all~$i$, see~\cite{Bra69}.

Computation of the Betti numbers with coefficients in~$\Q$ of a small cover~$M_{P,\lambda}$ generally is a hard problem. There is an unpublished formula by Suciu and Trevisan that reduces this computation to the computation of the homology groups for certain unions of faces of~$P$. For small covers of graph-associahedra~$M_{P_{\Gamma},\clambda}$  corresponding to the canonical characteristic function~$\clambda$ coming from the Delzant structure, Choi and Park~\cite{ChPa15} obtained an explicit formula for the Betti numbers with coefficients in~$\Q$ from special invariants of the graph~$\Gamma$. However, even in this case the problem of finding a graph with the smallest total Betti number of~$M_{P_{\Gamma},\clambda}$ remains unsolved, though it is very likely that the minimum is attained for the path graph~$L_{n+1}$. We shall restrict ourselves to considering three examples for which the Betti numbers with coefficients in~$\Q$ can be computed explicitly:

1. The Tomei manifold~$M^n_0=M_{\Pe^n,\lambda_0}$. This manifold is orientable. Its Betti numbers with coefficients in~$\Q$ were computed by Fried~\cite{Fri86}. They are equal to the \textit{Eulerian numbers of the first kind}:  $\beta_{i}^{\Q}(M^n_0)=A(n+1,i)$. Recall that the Eulerian number of the first kind~$A(m,k)$ is the number of permutations of numbers from~$1$ to~$m$  with exactly $k$ ascents,  that is, permutations $(\nu_1,\ldots,\nu_m)$ such that there are exactly $k$ indices~$s>1$ for which $\nu_s>\nu_{s-1}$. The total Betti number of the Tomei manifold is equal to $(n+1)!$\,. Notice that the $h$-numbers of the permutohedron~$\Pe^n$ are also equal to the Eulerian numbers of the first kind, $h_i(\Pe^n)=A(n+1,i)$, see~\cite{PRW08}, \cite{Buc08}. As we have already mentioned, the Betti numbers with coefficients in~$\Z_2$ of a small cover~$M_{P,\lambda}$ are independent of~$\lambda$ and equal to the $h$-numbers of~$P$. In particular,  $\beta_{i}^{\Z_2}(M^n_0)=\beta_{i}^{\Q}(M^n_0)=A(n+1,i)$. The coincidence of the Betti numbers with coefficients in~$\Q$ and in~$\Z_2$ is a specific property of the Tomei manifolds. Generally, the integral homology of a small cover often contains $2$-torsion, which implies that their Betti numbers with coefficients in~$\Q$ are smaller than their Betti numbers with coefficients in~$\Z_2$, that is, the $h$-numbers of~$P$.

2. The second example is also a small cover of the permutohedron but corresponding to the canonical Delzant characteristic function~$\clambda$. This manifold $M_{\Pe^n,\clambda}$ is the set of real points of a smooth projective toric variety, which is called the \textit{Hessenberg variety}. For $n\ge 2$, the manifold~$M_{\Pe^n,\clambda}$ is non-orientable. Its Betti numbers with coefficients in~$\Q$ were computed by Henderson~\cite{Hen12} (see also~\cite{ChPa15}):
$$
\beta_{i}^{\Q}(M_{\Pe^n,\clambda})=\binom{n+1}{2i}E_{2i},
$$
where $E_{m}$ is the \textit{Euler zigzag number}, that is, the number of permutations $(\nu_1,\ldots,\nu_m)$ such that
$\nu_1<\nu_2>\nu_3<\nu_4>\cdots$ (the signs~$<$ and~$>$ alternate). Equivalently, $E_m$ are the coefficients in the decomposition
$$
\sec t+\tan t=\sum_{m=0}^{\infty}\frac{E_m}{m!}\,t^m.
$$
Hence, the Betti numbers of the two-fold orientation covering~$\overline{M}_{\Pe^n,\clambda}$ of~$M_{\Pe^n,\clambda}$ are given by
\begin{gather*}
\beta_{i}^{\Q}(\overline{M}_{\Pe^n,\clambda})=\binom{n+1}{2i}E_{2i}+\binom{n+1}{2n-2i}E_{2n-2i},\\
\beta^{\Q}(\overline{M}_{\Pe^n,\clambda})=2\sum_{i=0}^{\left[\frac{n+1}{2}\right]}\binom{n+1}{2i}E_{2i}.
\end{gather*}

3. Now, consider the small cover~$M_{\As^n,\clambda}$. For $n\ge 2$, it is also non-orientable. Its Betti numbers with coefficients in~$\Q$ were computed by Choi and Park~\cite{ChPa15}:
$$
\beta_{i}^{\Q}(M_{\As^n,\clambda})=\left\{
\begin{aligned}
&\binom{n+1}{i}-\binom{n+1}{i-1},&&&0\le i\le \left[\frac{n+1}{2}\right]&,\\
&0,&&&i>\left[\frac{n+1}{2}\right]&.
\end{aligned}
\right.
$$ 
Hence, the Betti numbers of the two-fold orientation covering~$\overline{M}_{\As^n,\clambda}$ of~$M_{\As^n,\clambda}$ are given by
\begin{gather*}
\beta_{i}^{\Q}(\overline{M}_{\As^n,\clambda})=\beta_{n-i}^{\Q}(\overline{M}_{\As^n,\clambda})=\binom{n+1}{i}-\binom{n+1}{i-1},\qquad{}\\ {}\hspace{8cm} 0\le i< \left[\frac{n+1}{2}\right],\\
\beta_k^{\Q}(\overline{M}_{\As^{2k},\clambda})=2\binom{2k+1}{k}-2\binom{2k+1}{k-1},\\
\beta_{k-1}^{\Q}(\overline{M}_{\As^{2k-1},\clambda})=\beta_{k}^{\Q}(\overline{M}_{\As^{2k-1},\clambda})=\binom{2k}{k}-\binom{2k}{k-2},\\
\beta^{\Q}(\overline{M}_{\As^{n},\clambda})=2\binom{n+1}{\left[\frac{n+1}{2}\right]}.
\end{gather*}

It is not hard to show that for $n\ge 3$,
$$
2\binom{n+1}{\left[\frac{n+1}{2}\right]}<2\sum_{i=0}^{\left[\frac{n+1}{2}\right]}\binom{n+1}{2i}E_{2i}<(n+1)!\,,
$$
that is,
$$
\beta^{\Q}(\overline{M}_{\As^{n},\clambda})<\beta^{\Q}(\overline{M}_{\Pe^{n},\clambda})<\beta^{\Q}(M_{\Pe^{n},\lambda_0}).
$$
\begin{remark}
In fact, it is easy to see that even a single middle Betti number $\beta'(n)=\beta_{[n/2]}^{\Q}(\overline{M}_{\Pe^{n},\clambda})$ grows as $n\to\infty$ much faster than the total Betti number  $\beta''(n)=\beta^{\Q}(\overline{M}_{\As^{n},\clambda})$. Indeed, $\beta'(n)= 2(n+1)E_n$ if $n$ is even, $\beta'(n)> E_{n+1}$ if $n$ is odd, $\beta''(n)<2^{n+2}$, and it is well known that
$$
E_{2k}\sim 8\sqrt{\frac{k}{\pi}}\left(\frac{4k}{\pi e}\right)^{2k},\qquad k\to\infty.
$$  
\end{remark}

\subsection{Betti numbers with coefficients in~$\Z_2$}
Let $P$ be an $n$-dimensional simple polytope. Then the \textit{$f$-vector\/} of~$P$ is the integral vector $(f_{-1}(P),f_0(P),\ldots,f_{n-1}(P))$ such that $f_k(P)$ is the number of $(n-k-1)$-dimensional faces of~$P$. (In particular, $f_{-1}=1$, since the polytope is considered as the only  $n$-dimensional face of itself.) The \textit{$h$-vector\/} of~$P$ is the integral vector $(h_0(P),\ldots,h_n(P))$ such that
\begin{multline*}
h_0(P)x^n+h_1(P)x^{n-1}+\cdots+h_n(P)={}\\f_{-1}(P)(x-1)^n+f_0(P)(x-1)^{n-1}+\cdots+f_{n-1}(P).
\end{multline*}

Davis and Januszkiewicz~\cite{DaJa91} computed the cohomology ring with coefficients in~$\Z_2$ of an arbitrary small cover~$M_{P,\lambda}$, and showed that
\begin{equation}\label{eq_DJ}
\beta_i^{\Z_2}(M_{P,\lambda})=h_i(P).
\end{equation}
Postnikov, Reiner, and Williams~\cite{PRW08} computed the $h$-vectors for many important classes of graph-associahedra. Buchstaber and Volodin~\cite{BuVo11} proved that for any connected graph~$\Gamma$ on the vertex set~$\{0,\ldots,n\}$, there are inequalities
\begin{equation}\label{eq_h_ineq}
h_i(P_{\Gamma})\ge h_i(\As^n)=\frac{1}{n+1}\binom{n+1}{i}\binom{n+1}{i+1}, \quad 1\le i\le n-1.
\end{equation}
Moreover, each of inequalities~\eqref{eq_h_ineq} becomes an equality only if $\Gamma\cong L_{n+1}$, i.\,e., $P_{\Gamma}\cong\As^n$.

Thus, if the associahedron~$\As^n$ admitted a characteristic function~$\lambda^*$ such that the manifold~$M_{\As^n,\lambda^*}$ were orientable, then the  number
$$
\beta^{\Z_2}(M_{\As^n,\lambda^*})=\frac{1}{n+1}\sum_{i=0}^n\binom{n+1}{i}\binom{n+1}{i+1}
$$
would be the smallest among the total Betti numbers with coefficients in~$\Z_2$ of all orientable small covers~$M_{P_{\Gamma},\lambda}$ corresponding to connected graphs~$\Gamma$ on the vertex set $\{0,\ldots,n\}$. However, even in this case it would not  be clear whether this number would be the smallest among the total Betti numbers with coefficients in~$\Z_2$ of all manifolds in the class~$\CC_n$, since the computation of the Betti numbers with coefficients in~$\Z_2$ of the two-fold coverings~$\overline{M}_{P_{\Gamma},\lambda}$ is a harder problem. The author does not know for which~$n$ there exist orientable small covers of~$\As^n$. As we have mentioned above, the small covers~$M_{\As^n,\clambda}$ are non-orientable for all $n\ge 2$. It is easy to check that all small covers of the pentagon~$\As^2$ are non-orientable. On the other hand, there are orientable small covers over the three-dimensional associahedron~$\As^3$. For instance, one can take the characteristic function~$\lambda^*$ given by
\begin{gather*}
\lambda^*(a_{\{0\}})=\lambda^*(a_{\{1\}})=b_1,\qquad \lambda^*(a_{\{2\}})=\lambda^*(a_{\{3\}})=b_2,\\
\lambda^*(a_{\{0,1\}})=\lambda^*(a_{\{1,2\}})=\lambda^*(a_{\{2,3\}})=b_3,\\
\lambda^*(a_{\{0,1,2\}})=\lambda^*(a_{\{1,2,3\}})=b_1+b_2+b_3.
\end{gather*}

For a non-orientable $n$-dimensional manifold~$M$, it is easy to obtain the following estimates for the Betti numbers with coefficients in~$\Z_2$ of its two-fold orientation covering~$\overline{M}$:
$$
\beta_i^{\Q}(M)+\beta_{n-i}^{\Q}(M)\le \beta_i^{\Z_2}(\overline{M})\le 2\beta_i^{\Z_2}(M).
$$
The first inequality follows from~\eqref{eq_Brash} and the inequality $\beta^{\Q}_i(\overline{M})\le\beta_i^{\Z_2}(\overline{M})$; the second inequality is obtained from the Gysin exact sequence of the covering  $p\colon\overline{M}\to M$
$$
\ldots\to H^i(M;\Z_2)\xrightarrow{p^*} H^{i}(\overline{M};\Z_2)\xrightarrow{p_!} H^i(M;\Z_2)\xrightarrow{\psi} H^{i+1}(M;\Z_2)\to\ldots
$$
In this exact sequence, $\psi$ is the multiplication by the first Stiefel--Whitney class~$w_1(M)$. 

For small covers, Davis and Januszkiewicz~\cite{DaJa91} wrote explicitly the cohomology ring~$H^*(M_{P,\lambda};\Z_2)$  and the first Stiefel--Whitney class~$w_1(M_{P,\lambda})$. Nevertheless, it is a still unsolved problem to compute the dimensions of the kernel and the cokernel of~$\psi$ to which the computation of the Betti numbers with coefficients in~$\Z_2$ of~$\overline{M}_{P,\lambda}$ is reduced. However, the inequalities $2\beta^{\Q}(M)\le \beta^{\Z_2}(\overline{M})\le 2\beta^{\Z_2}(M)$ imply at least that the total Betti number with coefficients in~$\Z_2$ of~$\overline{M}_{\As^n,\clambda}$ grows much slower than the total Betti numbers with coefficients in~$\Z_2$ of the Tomei manifold~$M^n_0=M_{\Pe^n,\lambda_0}$ and of the two-fold covering of the real Hessenberg manifold~$\overline{M}_{\Pe^n,\clambda}$. (For the Tomei manifold formulae~\eqref{eq_DJ} easily imply that $\beta_i^{\Z_2}(M_0^n)=\beta_i^{\Q}(M_0^n)=A(n+1,i)$ and $\beta^{\Z_2}(M_0^n)=\beta^{\Q}(M_0^n)=(n+1)!$\,.)

\subsection{Simplicial volume}\label{subsection_simp_vol}
For any topological space~$X$, the vector spaces~$C_n(X;\R)$ of the singular chains of it with real coefficients can be endowed with the $L^1$-norms~$\|{\cdot}\|_1$ such that $\|\xi\|_1=\sum_{i=1}^q|\alpha_i|$ if $\xi=\sum_{i=1}^q\alpha_i\sigma_i$, where $\alpha_i$ are real numbers and $\sigma_i$ are pairwise different singular simplices. 
By definition, the \textit{simplicial volume\/} of an $n$-dimensional oriented closed manifold~$M$ is the number 
$$
\|M\|=\inf_{\xi\in[M]}\|\xi\|_1,
$$
where the infimum is taken over all singular cycles $\xi\in C_n(M;\R)$ representing the fundamental homology class $[M]\in H_n(M;\R)$.

If~$M^n$ dominates~$N^n$, then $\|M^n\|\ge \|N^n\|$. It is well known that $\|\hM^n\|=r\|M^n\|$ whenever $\hM^n$ is an $r$-fold covering of~$M^n$. Hence, for $n\ge 2$, any  URC-manifold has a nonzero simplicial volume.

\begin{problem}\label{problem_sv}
In every dimension $n\ge 3$ find the infimum of the simplicial volumes of $n$-dimensional URC-manifolds. If this infimum is achieved, find a URC-manifold~$M^n$ with the smallest simplicial volume~$\|M^n\|$. 
\end{problem}

By a well-known theorem of Gromov (see~\cite{Gro82}), for an $n$-dimensional hyperbolic manifold, we have $\|M^n\|=\mvol(M^n)/v_n$, where $\mvol(M^n)$ is the hyperbolic volume of~$M^n$, and $v_n$ is the supremum of volumes of convex simplicies in the $n$-dimensional Lobachevskii space~$\Lambda^n$. In particular, $v_3=1.0149\ldots$ is the volume of a regular ideal tetrahedron in~$\Lambda^3$. It is well known that the closed three-dimensional hyperbolic manifold with the smallest hyperbolic volume is the so-called Fomenko--Matveev--Weeks manifold~$Q_1$. Its volume $\mvol(Q_1)=0.9427\ldots$ was computed by Matveev and Fomenko~\cite{MaFo88}. The minimality of its volume was conjectured in~\cite{MaFo88} and was proved by Gabai, Meyerhoff,  and Milley~\cite{GMM09}. As we have mentioned above, all three-dimensional hyperbolic manifolds are URC-manifolds. In particular, $Q_1$ is a URC-manifold. Its simplicial volume is 
$\|Q_1\|=\mvol(Q_1)/v_3=0.9288\ldots$

\begin{quest}\label{quest_sv}
Is it true that the Fomenko--Matveev--Weeks manifold has the minimal simplicial volume among all three-dimensional URC-manifolds?
\end{quest}

Unfortunately, it is very hard to compute or at least to estimate the simplicial volumes of manifolds. The author knows no approaches to Question~\ref{quest_sv} and Problem~\ref{problem_sv}. Other results on relationship between URC-manifolds and simplicial volume can be found in the author's paper~\cite{Gai13-b}.

\smallskip

The author is grateful to V.\,M.~Buchstaber for multiple fruitful discussions.

\end{document}